\documentclass[a4paper,12pt]{amsart}
\usepackage{latexsym}
\usepackage{amssymb}
\usepackage{amsthm}

\RequirePackage{color}
\usepackage[normalem]{ulem}

\numberwithin{equation}{section}
	
\newtheorem{thm}{Theorem}[section]

\newtheorem{lemma}[thm]{Lemma}
\newtheorem{prop}[thm]{Proposition}
\newtheorem{cor}[thm]{Corollary}

\theoremstyle{definition}
\newtheorem{definition}[thm]{Definition}

\theoremstyle{remark}
\newtheorem{remark}[thm]{Remark}

\newtheorem{remarks}[thm]{Remarks}
\newtheorem{questions}[thm]{Questions}

\def\go{G^{(0)}}
\def\leqa{\leq_{\text{ring}}}
\def\lessa{<_{\text{ring}}}
\def\sima{\sim_{\text{ring}}}
\def\cba{\precsim_{\text{ring}}}
\def\cbc{\precsim_{C^{\ast}}}
\def\cs{C^{\ast}}
\DeclareMathOperator{\KP}{KP}

\newcommand{\N}{\mathbb{N}}

\newcommand{\inv}{^{-1}}
\newcommand{\C}{\mathbb{C}}

\title[Purely infinite and simple Steinberg algebras groupoid
$\cs$-algebras]{Dense subalgebras of purely infinite simple groupoid $\cs$-algebras}

\author[J.H. Brown]{Jonathan H. Brown}
\address[J.H. Brown]{
Department of Mathematics\\
University of Dayton\\
300 College Park Dayton\\
OH 45469-2316 USA} \email{jonathan.henry.brown@gmail.com}

\author[L.O. Clark]{Lisa Orloff Clark}

\author[A. an Huef]{Astrid an Huef}

\address[L.O. Clark and A. an Huef]{School of Mathematics and
Statistics, Victoria University of Wellington, P.O. Box 600, Wellington 6140, New Zealand}
\email{lisa.clark@vuw.ac.nz}
\email{astrid.anhuef@vuw.ac.nz }

\subjclass[2010]{46L05, 16S60}
\keywords{Purely infinite ring, purely infinite $C^*$-algebra, ample groupoid, Steinberg algebra, Kumjian--Pask algebra,   infinite projection, infinite idempotent.}
\thanks{Part of this research was done while the first named author was visiting the other two authors: he would like to thank them for their hospitality during this visit.  This research was supported by the Marsden grant   from the Royal Society of New Zealand and by the Seed grant IGRS03  from the University of Dayton Research Institute.  We thank an anonymous referee for helpful suggestions, in particular for showing us Lemma~\ref{lem:referee}. }
\date{February 2020}

\begin{document}

\begin{abstract}	
A simple Steinberg algebra associated to an ample Hausdorff groupoid $G$ is 
algebraically purely infinite  if and only if the characteristic functions of compact open 
subsets of the unit space are infinite idempotents.  If a simple Steinberg algebra  is algebraically purely infinite, then the reduced groupoid $C^*$-algebra $C^*_r(G)$ is simple  and purely infinite. But the Steinberg algebra seems too small for the converse to hold.  For this purpose we introduce an intermediate $*$-algebra $B(G)$ constructed using corners $1_U C^*_r(G) 1_U$ for all compact open subsets $U$ of the unit space of the groupoid.  We then show that if $G$ is minimal and effective, then $B(G)$ is algebraically properly infinite  if and only if $C^*_r(G)$ is purely infinite simple. We apply our results to the algebras of higher-rank graphs.
\end{abstract}
 \maketitle

\section{Introduction}
Inspired by the success of classifying purely infinite simple $C^*$-algebras, Ara, Goodearl and Pardo  
extended the notion of  purely infinite $C^*$-algebras to unital rings, and studied their  basic properties \cite{AGP}.    
Purely infinite rings were   studied further in \cite{Ara, AA, GP, AGPS}.  
In particular,  \cite{AGPS} develops a general theory of purely infinite rings by presenting an 
algebraic parallel of the work in \cite{KR} on purely infinite $\cs$-algebras. Here we investigate a question asked in \cite[Problem~8.4]{AGPS}:
\begin{itemize}
\item[]   Suppose that $A_0$ is a dense subalgebra of a $C^*$-algebra $A$. If $A_0$ is purely infinite in the algebraic sense, is  $A$ purely infinite in the $C^*$-algebraic sense?
\end{itemize}
We focus on the situation where the $\cs$-algebra is a simple $\cs$-algebra associated to an ample Hausdorff groupoid, that is, an \'etale groupoid with a basis of compact open sets.   Since every non-unital, purely infinite, simple, nuclear $\cs$-algebra in the  UCT class is isomorphic to the reduced $\cs$-algebra $C^*_r(G)$ of some ample Hausdorff groupoid $G$ (see \cite[p. 367]{Sp}), 
this is a good setting to investigate this problem and its converse.

We investigate two different dense subalgebras of $\cs_r(G)$.  The first is the complex Steinberg algebra $A(G)$ associated to an ample groupoid $G$.   Simplicity of $A(G)$ and $C^*_r(G)$ are well understood (provided $G$ is amenable and second countable): $A(G)$ is algebraically simple if and only if $C^*_r(G)$ is $C^*$-algebraically simple, and this happens if and only if $G$ is minimal and effective \cite{BCFS}.  
For the remainder of the introduction,  we assume that $G$ is a minimal and effective ample groupoid.  

It was shown in \cite[Theorem~4.1]{BCS} that $C^*_r(G)$ is purely infinite if and only if, for every compact open set $U$ of the unit space, the characteristic function $1_U$ is an infinite projection.  As a direct consequences, we obtain an affirmative answer to \cite[Problem~8.4]{AGPS} when $A_0 =A(G)$ (see Corollary~\ref{cor:main}).  The converse remains an open question because, in a way, $A(G)$ is too small.

The second dense subalgebra $B(G)$ of $C^*_r(G)$ that we investigate contains the corners  $1_U C^*_r(G) 1_U$ for all compact open subsets $U$ of the unit space of the groupoid, and so  is much larger than $A(G)$.  
The reason $B(G)$ is useful is two-fold: First, unlike $C^*_r(G)$, $B(G)$ is s-unital, and so we can apply the results of \cite{AGPS} directly to it.  Second, unlike $A(G)$, $B(G)$ has the property that $B(G)xB(G)$ is contained in $B(G)$ for every $x\in C^*_r(G)$.   Using these two properties, we show in Theorem~\ref{thm:main} that $B(G)$ is algebraically properly infinite if and only if $\cs(G)$ is purely infinite.

Our motivation for studying these dense subalgebras is to make progress on the long-standing open question of finding necessary and sufficient conditions on  a groupoid $G$ that characterise when the $C^*_r(G)$ is purely infinite.   
Anantharaman-Delaroche introduced the notion of a locally contracting groupoid in~\cite{AD}, and then showed the reduced $C^*$-algebra of a locally contracting groupoid is purely infinite  \cite[Proposition~2.4]{AD}.  It is not known if the converse holds, even for the $C^*$-algebras of higher-rank graphs which can be modelled using ample groupoids.  

It follows from \cite[Proposition~2.2]{AD} that in an ample, locally contracting  groupoid,  the $1_U$ are all infinite projections, but again it is not known if the converse holds. For the groupoids of higher-rank graphs, it would suffice 
to understand what graph properties ensure that all vertex projections are
infinite projections in the graph $C^*$-algebra \cite[Corollary~5.1]{BCS}.

Following \cite{CaHS}, there has been substantial progress made in investigating a possible dichotomy for the $C^*$-algebras of ample groupoids  \cite{BL, RS}: that $C^*$-algebras of ample groupoids  are all either purely infinite or stably finite. At this stage, it has been established that the dichotomy holds when a certain type semigroup is almost unperforated.   Also, Corollary~4.9 of \cite{BL} gives a sufficient condition for an inner exact groupoid to have a purely infinite $C^*$-algebra. Proposition~7.1 of \cite{RS}  gives a necessary condition on $G$ for the $C^*$-algebra to be purely infinite.  How either of these conditions relate to  a locally contracting groupoid is not immediately clear.

The outline of this paper is as follows.  In Section~\ref{sec-prelim} we give some preliminaries, 
including algebraic and $C^*$-algebraic definitions of infinite idempotents/projections, and purely infinite rings and $C^*$-algebras.
The main focus of Section~\ref{section:1U} is to 
prove an algebraic version of \cite[Theorem~4.1]{BCS} for Steinberg algebras: 
$A(G)$ is algebraically purely infinite if and only if, for every compact open 
set $U$ of the unit space, $1_U$ is an infinite idempotent  in  $A(G)$.  
In Section~\ref{section:main}, we introduce $B(G)$.  Our main theorem, Theorem~\ref{thm:main}, says that $B(G)$ is algebraically properly infinite if and only if $\cs_r(G)$ is $\cs$-purely infinite.  The forward implication is very general in that it does not require $G$ to be minimal or effective.

Finally, in Section~\ref{section:KP}, we show how our results apply
to higher-rank graphs and their algebras. The Steinberg algebra of the graph groupoid is canonically isomorphic to the  Kumjian--Pask algebra of the graph \cite{ACaHR,CP}, and the $C^*$-algebra of the graph groupoid is isomorphic to the $C^*$-algebra of the graph.  We reconcile our  results with recent results by Larki about purely infinite Kumjian--Pask algebras \cite{Lar}.
Finally, in Corollary~\ref{cor:KP}, we consider the special case of an aperiodic and cofinal higher-rank graph which
has a return path with an entrance, and show these graphs give rise to purely infinite algebras.
Along the way, we prove Lemma~\ref{lem:ret} which allows us to construct return paths with entrances. This 
fills a gap in \cite[Corollary~5.7]{EvansSims:JFA2012} and \cite[Proposition~8.8]{Sims}.

\section{Preliminaries}\label{sec-prelim}

We start by summarising the algebraic and $\cs$-algebraic definitions that we use from \cite{AGPS} and \cite{KR}, respectively.

\subsection*{Algebraic preliminaries}
 \begin{definition}\label{defn ring} Suppose that $p$ and $q$ are idempotents in a ring $R$, $a \in R$, 
$x \in M_k(R)$ and $y \in M_n(R)$.  Set $RaR:=\text{span}\{cad:c,d\in R\}$.
\begin{enumerate}
\item  $R$ is {\em s-unital}, if for each $r \in R$, there exist $u,v \in R$ such that $ur=rv=r$;
\item  $p \leqa q$, if $pq=qp=p$;
\item $p \lessa q$ if $p \leqa q$ and $p \neq q$;
\item  $p \sima q$ if there exist $s, t \in R$ such that $st=p$ and $ts =q$;
\item  $x \cba y$ if there exists $\alpha \in M_{kn}(R)$ and $\beta \in M_{nk}(R)$ such that  $x=\alpha y \beta$;
\item  $p$ is an {\em infinite  idempotent}  if there exists an idempotent $q$ such that $p \sima q \lessa p$;
\item  $p$ is a {\em  algebraically properly infinite idempotent} if  there exists an idempotent $q$ such that $p \oplus p \sima q \leqa p$;
\item  $a$ is {\em algebraically infinite} 	if  there exists  $t \in R\setminus\{0\}$ such that $a \oplus t \cba a$;
\item  $a$ is {\em algebraically properly infinite} if  $a\neq 0$ and $a \oplus a \cba a$;
\item $R$ is {\em algebraically  purely infinite simple} 	if  $R$ is algebraically simple ($R$ has no nontrivial two-sided ideals) and  every nonzero right ideal  contains an infinite idempotent;
\item  $R$ is  {\em algebraically  properly infinite}	(also called \emph{properly purely infinite}) if  every nonzero element of $R$ is algebraically properly infinite;
\item  $R$ is {\em algebraically 
purely infinite} 		if  
\begin{enumerate}
\item no quotient of $R$ is a division ring; and 
\item \label{b} whenever $a \in R$ and $b \in RaR$, we have $b \cba a$.
\end{enumerate}
\end{enumerate}
\end{definition}

\begin{remarks} Here we  clarify some of the relationships between the definitions given above.
\label{rmk resolve ring}
\begin{enumerate}
\item\label{it1:rmkresolve} If $p$ and $q$ are idempotents with $p\leqa q$, then $p\cba q$.  
\item By \cite[Remarks~2.7]{AGPS} an idempotent is an infinite idempotent if and only if it is algebraically infinite and it is an algebraically properly infinite idempotent if and only if it is algebraically properly infinite.
\item Let $R$ be a ring with local units which means that for every finite set $F$ there exists an idempotent $e$ such that $F\subset eRe$. (The existence of local units implies that  $R$ is s-unital.) Then $R$ is algebraically purely infinite simple  if and only if (i) $R$ is not a division ring and (ii) for all $a,b\in R$ with $a\neq 0$ we have $b \cba a$ \cite[Proposition~10]{AA}; in particular,  $R$ is simple and  algebraically purely infinite if and only if  $R$ is algebraically purely infinite simple.
\end{enumerate}
\end{remarks}

The next lemma is an algebraic version, for s-unital rings, of \cite[Lemma~3.17]{KR}.

\begin{lemma}
\label{lem:3.9}
Suppose that $R$ is an s-unital ring.   Let $p, q, a\in R$, and suppose that  $q$ is an idempotent, and that $p$ is an infinite idempotent.
\begin{enumerate}
\item\label{it1:lem3.9}  If $p \cba a$, then $a$ is algebraically infinite.
\item\label{it2:lem3.9}  If $p \leqa q$, then $q$ is an infinite idempotent.
\item\label{it3:lem3.9}  If $p \sima q$, then $q$ is an infinite idempotent.
\end{enumerate}
\end{lemma}

\begin{proof}\eqref{it1:lem3.9} is precisely \cite[Lemma~3.9(iii)]{AGPS}.
For \eqref{it2:lem3.9}, suppose that $p \leqa q$. By Remark~\ref{rmk resolve ring}\eqref{it1:rmkresolve}, $p\cba q$.   Since $p$ is algebraically infinite,  \eqref{it1:lem3.9} gives that $q$ is also algebraically infinite.

For \eqref{it3:lem3.9}, suppose  $p \sima q$.
Then there exist $r,s \in R$ such that $rs=p$ and $sr =q$.
Let $\alpha = r$ and $\beta = s$.  
Then
\[\alpha q \beta = rqs=r(sr)s = (rs)^2 = p^2 = p.\]
Now $p\leqa q$ and $p$ is algebraically infinite, and so $q$ is algebraically infinite by  \eqref{it1:lem3.9}.  
\end{proof}


\subsection*{$C^*$-algebraic preliminaries}  Let $A$ be a $\cs$-algebra. We write $A^+$ for the positive elements of $A$, and then $A^+=\{a^*a: a\in A\}$. We write $C_0(X)$ for the $C^*$-algebra of continuous functions from $X$ to $\mathbb{C}$ that vanish at infinity. The following definitions come from \cite{KR}.

\begin{definition}\label{C* def} Let $A$ be a $C^*$-algebra, let $p$ and $q$ be projections in $A$, and let $a, b \in A^{+}$, $x \in M_k(A)$ and $y \in M_n(A)$.
\begin{enumerate}
\item  
$a \leq_{\cs} b$ 	if $b-a\in A^{+}$;
\item  
$a <_{\cs} b$ if $a \leq_{\cs} b$ and $a \neq b$;
\item 
$p \sim_{\cs} q $ 	if  there exists $t \in A$ such that $tt^* = p$  and $t^*t = q$;
\item 
$x \cbc y$  if there exists a sequence $\{\alpha_i\} \subseteq M_{k,n}(A)$  such that $\alpha_iy\alpha_i^* \to x$ in norm;
\item 
$p$ is an {\em infinite  projection 	}	if  there exists a projection $q$ such that $p \sim_{\cs} q <_{\cs} p$;
\item $p$ is a {\em $\cs$-properly infinite projection} if there exist mutually orthogonal projections $q_1$ and $q_2$ such that $q_1 \leq p, q_2\leq p$ and $q_1 \sim_{\cs} p \sim_{\cs} q_2$;
\item 
$a$ is {\em $\cs$-infinite} 	if there exists a nonzero  $t \in A^{+}$ such that $a \oplus t \cbc a$;
\item 
$a$ is {\em $\cs$-properly
infinite}	if  $a\neq 0$ and $a \oplus a \cbc a$;

\item A $C^*$-subalgebra $B\subseteq A$ is {\em hereditary} if $a\leq_{\cs} b$ and $b\in B$ implies $a\in B$. 
\item  
$A$ is {\em $\cs$-purely  infinite simple} if $A$ is simple ($A$ has no nontrivial closed ideals) and  every hereditary subalgebra  contains an infinite projection; 
\item 
\label{A pu inf} 
$A$ is {\em $\cs$-purely infinite} if
\begin{enumerate}   
\item there are no nonzero $*$-homomorphisms from $A$ to $\mathbb{C}$; and 
\item for every $a,b \in A^+$ we have $b \cbc a$ if and only if $b \in \overline{AaA}$.
\end{enumerate}
\end{enumerate}
\end{definition}

\begin{remarks} As in the algebraic situation there are relationships between these notions.
\begin{enumerate}
\item   By \cite[Lemma~3.1]{KR}, a projection $p$ is a $\cs$-infinite projection  if and only if it is $\cs$-infinite.
Similarly $p$ is a  $\cs$-properly infinite projection if and only if $p$ is $\cs$-properly infinite.
\item    By Theorem~4.16 of \cite{KR}, a $\cs$-algebra 
$A$ is $\cs$-purely infinite if and only if every nonzero positive element of $A$ is $\cs$-properly infinite.  There is no algebraic analogue of \cite[Theorem~4.16]{KR}; see \cite[Examples~3.5]{AGPS} for a counter example.  
\item By \cite[Proposition~5.4]{KR} a $C^*$-algebra $A$ that is  simple and $\cs$-purely infinite   is $C^*$-purely infinite simple.  By \cite[Proposition~4.7]{KR} a $C^*$-algebra $A$ is $\cs$-purely infinite simple  implies that $A$ is simple  and $A$ is $\cs$-purely infinite.  Thus the two notions of $\cs$-purely infinite are compatible for simple $C^*$-algebras. 
\end{enumerate}
\end{remarks}

\begin{lemma}
\label{lem:referee}
Let $A$ be a $\cs$-algebra and $A_0$ a dense subalgebra.
Let $p$ be a projection in $A_0$.  If $p$ is an infinite idempotent in $A_0$, then $p$ is an infinite projection in $A$.\end{lemma}

\begin{proof}
Since $p$ is an infinite idempotent in $A_0$, there exists an idempotent $e \in A_0$
such that $p \sima e \lessa p$.  Thus $eA \subsetneq pA.$
By \cite[Exercise~3.11(1)]{RLL} there exists a projection $q \in A$  such that $e \sima q$ in $A$ and $eA = qA$.  
Since $\sima$ is transitive, we have $p \sima q$ in $A$.  Thus we have $p \sim_{\cs} q$ by, for example, \cite[Exercise~3.11(2)]{RLL}.
Also, since $pA \neq eA = qA$,  we have $q \neq p$.  
To see that $q \leq_{\cs} p$, we show that $pq=q$, which suffices by \cite[Theorem~2.3.2]{Murphy}.  Since $qA \subsetneq pA$, there exists $a \in A$ such that $q^2 = pa$, that is $q =pa$.  So $pq = p^2a = pa = q.$ Thus $p$ is an infinite projection.
\end{proof}

\subsection*{Groupoids and groupoid algebras} 

Let $G$ be a topological groupoid.   We write $\go$ for the unit space of $G$, and 
$r, s: G\to \go$ for the range and source maps.  An open subset $B$ of $G$ is an \emph{open bisection} 
if the source and range maps restrict to homeomorphisms on $B$, and  $r(B)$ and $s(B)$ are open subsets of $\go$. 
We say that $G$ is \emph{ample} if $G$ has a basis of compact open bisections.  Equivalently, $G$ is 
ample if  $G$ is \'etale and its unit space $\go$ is totally disconnected  (see \cite[Proposition~4.1]{Exel2}).  
In this paper, we consider only ample Hausdorff groupoids.
 
Let $G$ be an ample Hausdorff groupoid and let $K$ be a field. The \emph{Steinberg algebra} $A_K(G)$ is the 
$K$-algebra of functions from $G$ to $K$ that are both locally constant and compactly supported.
Addition is defined pointwise. The multiplication is convolution: for $f,g\in A_K(G)$ we have
\[(f * g)(\gamma) = \sum_{\alpha\beta = \gamma} f(\alpha)g(\beta).\]
 When $K$ has an involution $k\mapsto \overline{k}$, we define an involution $f\mapsto f^*$ on $A_K(G)$ by
\[(f^*)(\gamma) = \overline{f(\gamma^{-1})}.\]
We write $A(G)$ for $A_{\mathbb{C}}(G)$.

Let $B$ be a compact open bisection, and write $1_B$ for the characteristic function from $B$ to $K$.
We will often use  that every $f$ in $A_K(G)$ can be expressed as a  sum
\[f = \sum_{B \in F} c_B1_B, \]
where $F$ is a finite set of disjoint compact open bisections and $0\neq c_B \in K$ for each $B\in F$.  
For compact open bisections $B$ and $C$ we have 
\[
1_B*1_C=1_{BC} \quad \text{and}\quad 1^*_B=1_{B\inv}.
\]
Indeed, the compact open bisections in an ample groupoid form an inverse semigroup under the operations
\[
BC=\{\gamma\eta:r(\eta)=s(\gamma), \gamma\in B, \eta\in C\}\quad\text{and}\quad B\inv=\{\gamma\inv:\gamma\in B\},
\]
and the multiplication in $A_K(G)$ agrees with these operations.  
Every Steinberg algebra is s-unital by \cite[Lemma~2.6]{CEP}.  See \cite{CFST} and \cite{Steinberg} for more information about Steinberg algebras.

The \emph{reduced groupoid $\cs$-algebra} $C^*_r(G)$ was introduced in \cite{Renault}, and for  
ample groupoids the definition reduces to the following. With convolution and involution as above, the set $C_c(G)$ of continuous compactly supported function from $G$ to $\C$ is a $*$-algebra. Let $u\in \go$. Set $G_u=\{\gamma\in G: s(\gamma)=u\}$.  Write $\delta_{\gamma}$ for the basis of point masses in $\ell^2(G_u)$. There is a $*$-representation $R_u : C_c(G) \to \mathcal{B}(\ell^2(G_u))$ such that, for $f \in C_c(G)$ and $\gamma \in G_u$,
\[
 R_u(f)\delta_{\gamma} = \sum_{\{\alpha\in G : s(\alpha)= r(\gamma)\}}
 f(\alpha)\delta_{\alpha\gamma}.
\]
The reduced $C^*$-algebra $\cs_r(G)$ is the completion of the image of $C_c(G)$ under the direct sum $\bigoplus_{u \in \go} R_u$.
By construction, the complex Steinberg algebra $A(G)$ is a dense $*$-subalgebra of $C^*_r(G)$.

We say that  $G$ is \emph{effective} if for every compact open bisection 
$B \subseteq G \setminus \go$, there exists $\gamma \in B$ such that
$s(\gamma) \neq r(\gamma)$.   (See \cite[Lemma~3.1]{BCFS} and \cite[Corollary~3.3]{Ren08}
 for some equivalent characterisations of effective groupoids.)
  We say that $G$ is \emph{minimal} if $\go$ has no nontrivial open invariant subsets.
By \cite[Theorem~4.1]{CE}, $A_K(G)$ is algebraically simple if and only 
if $G$ is  minimal and effective.  
If $G$ is second-countable,  minimal and effective, 
then $\cs_r(G)$ is $\cs$-simple, and the converse holds if $G$ is amenable (see, for example, \cite[Theorem~5.1]{BCFS}).

Since $G$ is ample, the unit space $\go$ is open in $G$, and we may view $C_0(\go)$ as a subalgebra of $\cs_r(G)$. For $f,g \in C_0(\go)$ we have 
\[f \leq_{\cs} g \text{ if and only if } f(x)\leq g(x) \text{ for all } x \in \go.\]

%

\section{Algebraically purely infinite simple Steinberg algebras}
\label{section:1U}

By combining \cite[Theorem~4.1]{BCS} and Lemma~\ref{lem:referee}, we get the following interesting result.

\begin{cor}
\label{cor:main}
 Let $G$ be a second-countable ample Hausdorff  groupoid. Suppose that $G$ is minimal and effective.  If
$A(G)$ is  algebraically purely infinite, then $\cs_r(G)$ is $\cs$-purely infinite.
\end{cor}

\begin{proof}
Since $G$ is minimal and effective, by \cite[Theorem~4.1]{BCS} it suffices to show that $1_U$ is an infinite projection for every compact open $U \subseteq \go$.	Since $A(G)$ is algebraically purely infinite simple, every idempotent is infinite \cite[Proposition~1.5]{AGP}.  Thus every $1_U$ with $U \subseteq \go$ is an infinite idempotent in $A(G)$ and hence the result follows from Lemma~\ref{lem:referee}.
\end{proof}

We now prove the algebraic analogue of Theorem~4.1 in \cite{BCS}.
\begin{thm}
\label{thm:pis} Let $K$ be a field and let $G$ be an ample Hausdorff groupoid. Suppose that $G$ is minimal and effective.  Then  $A_K(G)$ is algebraically purely infinite simple if and only if 
$1_U$ is an infinite idempotent for every compact open $U \subseteq \go$.
\end{thm}

\begin{proof} By \cite[Theorem~4.1]{CE}, $A_K(G)$ is algebraically simple if and only 
if $G$ is  minimal and effective.  

First suppose that $A_K(G)$ is algebraically purely infinite simple. Then all nonzero idempotents are infinite by \cite[Proposition~1.5]{AGP}. In particular,  $1_U$ is an infinite idempotent for every compact open $U \subseteq \go$.

Conversely, assume $1_U$ is an infinite idempotent for every compact open $U \subseteq \go$.
Fix a nonzero right ideal $I \subseteq A_K(G)$. We need to show that $I$ contains an infinite idempotent. Let $0 \neq b\in I$ and write
\[b = \sum_{B \in F} c_B1_B\]
where $F$ is a finite  set of disjoint compact open bisections and $0\neq c_B\in K$  for each $B\in F$.  Fix $B_0 \in F$.
Then 
\[
 c_{B_0}^{-1}b1_{B_0^{-1}} = 1_{r(B_0)} + \sum_{B \in F, B \neq B_0} c_{B_0}^{-1}c^{}_B1_{BB_0^{-1}}.
\]
Since $F$ is a disjoint collection,  for each $B \neq B_0$ we have $BB_0^{-1} \subseteq G \setminus \go$.
So 
\[L:= \bigcup_{B \in F, B \neq B_0} BB_0^{-1} \]
is a compact open subset of $G \setminus \go$.
Let $U:=r(B_0)$. Then $U$ is a compact open subset of $\go$.  
Since $G$ is effective, we apply \cite[Lemma~3.1(4)]{BCFS}
to get a nonempty open set $V \subseteq U$ such that $VLV = \emptyset$.
Now 
\begin{align*}
1_V(  c_{B_0}^{-1}b1_{B_0^{-1}})1_V
&= 1_V (1_{r(B_0)} + \sum_{B \in F, B \neq B_0} c_{B_0}^{-1}c^{}_B1_{BB_0^{-1}})1_V\\
&= 1_V + \sum_{B \in F, B \neq B_0} c_{B_0}^{-1}c^{}_B1_{VBB_0^{-1}V}\\
&=1_V.
\end{align*}
Set $p:= (  c_{B_0}^{-1}b1_{B_0^{-1}})1_V$. Then $p\in I$ and  \[p^2 = (c_{B_0}^{-1}b1_{B_0^{-1}})\left(1_V(  c_{B_0}^{-1}b1_{B_0^{-1}})1_V\right)
 = (  c_{B_0}^{-1}b1_{B_0^{-1}}) 1_V  = p.\]
Thus  $p$ is a nonzero idempotent in $I$. Further, 
$1_Vp = 1_V$ and $p1_V = p$, and hence $p \sima 1_V$. 
\end{proof}

\begin{cor}
 \label{cor:basis}
 Let $K$ be a field, let $G$ be an ample Hausdorff groupoid and let $\mathcal{B}$ be a basis for $\go$ consisting of compact open subsets. Suppose that $G$ is  minimal and effective.   Then $A_K(G)$ is algebraically purely infinite simple if and only if  $1_U$ is an algebraically infinite idempotent for every compact open $U \in \mathcal{B}$. 
\end{cor}

\begin{proof}
The forward implication is immediate from Theorem~\ref{thm:pis}.  For the reverse implication, suppose  that
$1_U$ is an algebraically infinite idempotent for every compact open $U \in \mathcal{B}$.
Fix a compact open subset $V \subseteq \go$.  By Theorem~\ref{thm:pis}, it suffices to show $1_V$ is an 
infinite idempotent.  Since $\mathcal{B}$ is a basis, there exists $W \in \mathcal{B}$ such
that $W \subseteq V$.  
Then
\[1_W1_V = 1_W = 1_V1_W,\]
and so $1_W \leqa 1_V$.  By assumption, $1_W$ is an algebraically infinite idempotent. Since $A_K(G)$ is s-unital, Lemma~\ref{lem:3.9}(\ref{it2:lem3.9}) applies, and 
$1_V$ is an algebraically infinite idempotent as well.  
\end{proof}

A Hausdorff ample groupoid $G$ is \emph{locally contracting} \cite[Definition~2.1]{AD} if for every compact open $U \subseteq \go$,  there exists a nonempty compact open bisection $B \subseteq G$  such that \[s(B) \subsetneq r(B) \subseteq U.\] 
The following corollary to Theorem~\ref{thm:pis} is the algebraic analogue of 
\cite[Proposition~2.4]{AD} which gives a useful sufficient condition for a $C^*$-algebra to be $\cs$-purely infinite simple. 

\begin{cor}
\label{cor pis} Let $K$ be a field and let $G$ be an ample Hausdorff groupoid. Suppose  that $G$ is  minimal, effective and locally contracting. Then $A_K(G)$ is algebraically purely infinite simple.
\end{cor}

\begin{proof} Fix $U\subseteq \go$.  By Theorem~\ref{thm:pis}, it suffices to show that $1_U$ is an infinite idempotent.
 Since $G$ is locally contracting, there exists a compact open bisection $B \subseteq G$  such that $s(B) \subsetneq r(B) \subseteq U$.
Then \[1_{s(B)}1_{r(B)}=1_{s(B)}=1_{r(B)}1_{s(B)}\] gives $1_{s(B)}\leqa 1_{r(B)}$, and  $s(B)\subsetneq r(B)$ gives  
$1_{s(B)}\lessa 1_{r(B)}$. Further, 
\[1_{B^{-1}}1_{B}=1_{B^{-1}B}=1_{s(B)} \quad \text{and} \quad 1_{B}1_{B^{-1}}=1_{BB^{-1}}=1_{r(B)}\] and so $1_{r(B)}\sima 1_{s(B)}$. Thus 
\[
1_{r(B)}\sima 1_{s(B)}\lessa 1_{r(B)}.
\]
 Therefore $1_{r(B)}$ is an infinite idempotent. Since $1_{r(B)}\leqa 1_U$, $1_U$ is algebraically infinite by Lemma~\ref{lem:3.9}\eqref{it2:lem3.9}. 
By \cite[Remarks~2.7]{AGPS}, $1_U$ is an algebraically infinite idempotent as required.  
\end{proof}

\section{A larger dense subalgebra}
\label{section:main}

We want to connect the notion of $\cs$-purely infinite simple in the $\cs$-algebra with some kind of algebraic pure infiniteness.  Corollary~\ref{cor:main} is a start but we seek an \emph{if and only if} characterisation and $A(G)$ seems too small.
Even with the assumption that $G$ is minimal and effective, if $\cs_r(G)$ does not have an identity, then it is not s-unital and it is not algebraically simple: for example, 
the Pedersen ideal (see \cite[Theorem~5.6.1]{ped})  is always a nontrivial dense nonclosed ideal.

We now introduce a large s-unital subalgebra of $\cs_r(G)$.
Define \[B(G):= \{1_U x 1_U : x \in \cs_r(G) \text{ and } U \subseteq \go \text{ is a compact open set}\}.\]
It is straightforward to check that $A(G)$ is a subset of $B(G)$.  Since the Pedersen ideal contains all the projections of $\cs_r(G)$ (see \cite[5.6.3]{ped}), it contains all elements of the form $1_U$ for compact open $U \subseteq \go$.  Thus $B(G)$ is contained in the Pedersen ideal.

We apply techniques from  \cite[Theorem~3.17]{AGPS} to see that when $G$ is minimal and effective, $B(G)$ is algebraically properly infinite if and only if $C^*_r(G)$ is $C^*$-purely infinite.  
\begin{lemma}\label{lem:B}  Let $G$ be an ample Hausdorff groupoid. 
Then $B(G)$ is a dense s-unital subalgebra of $\cs_r(G)$.
\end{lemma}

\begin{proof} That $B(G)$ is dense in $C^*_r(G)$ follows because $A(G) \subseteq B(G)$ and $A(G)$ is dense in $\cs_r(G)$ by \cite[Proposition~4.2]{CFST}.
 For any $1_Ux1_U \in B(G)$ we have  $1_U(1_Ux1_U)1_U=1_Ux1_U$, and so $B(G)$ is s-unital.  
To see that $B(G)$ is a subalgebra, fix $b_1:=1_{U_1}x1_{U_1}$ and $b_2:=1_{U_2}x1_{U_2}$
in $B(G)$ and and $r\in \mathbb{C}$. Let $U:= U_1 \cup U_2$. Then $U$ is compact and open because the $U_i$ are.  Then
\[b_1-b_2=1_U(1_{U_1}x1_{U_1}-1_{U_2}x1_{U_2})1_U\in B(G).\]
Similarly, \[b_1b_2=1_U(1_{U_1}x1_{U_1}1_{U_2}x1_{U_2})1_U\in B(G) \quad \text{and} 
            \quad rb_1= r1_{U_1}x1_{U_1}=1_{U_1}(rx)1_{U_1} \in B(G).\qedhere\]
\end{proof}

\begin{remark}
If $\go$ is compact, then $\cs_r(G)$ is unital with unit $1_{\go}$.  Thus for any $x \in \cs_r(G)$, $x = 1_{\go}x1_{\go}$ and hence $B(G) = \cs_r(G)$.	
\end{remark}

\begin{prop}
\label{prop:1}
Let $G$ be a second-countable ample Hausdorff groupoid.  Suppose that $G$ is minimal and effective.  If $B(G)$ is algebraically properly infinite, then $\cs_r(G)$ is $\cs$-purely infinite.	
\end{prop}

\begin{proof}
We show that $\cs_r(G)$ is $\cs$-purely infinite by verifying the two conditions  of  Definition~\ref{C* def}\eqref{A pu inf}.  First, we have to show that there are no nonzero $C^*$-homomorphisms from $C^*_r(G)$ to $\C$. 
Aiming for a contradiction, suppose that there is
a nonzero $C^*$-homomorphism $\pi:\cs_r(G) \to \mathbb{C}$. Since $G$ is minimal and effective, $C^*_r(G)$ is simple \cite[Proposition~II.4.6]{Renault}, so $\pi$ is injective, and $C^*_r(G)$ is $C^*$-isomorphic to $\C$. In particular, $A(G)$ is isomorphic to $\C$. But $A(G)$ is algebraically purely infinite, and hence contains an infinite idempotent. But then so does $\C$, a contradiction. Hence  there are no nonzero homomorphisms from $C^*_r(G)$ to $\C$. 

Second, we have to show that for $x,y\in C_r^*(G)^+$, we have $y \cbc x$ if and only if $y\in \overline{C_r^*(G)x C_r^*(G)}$. Since $C_r^*(G)$ is simple, $\overline{C_r^*(G)x C_r^*(G)}=C_r^*(G)$; so
it suffices to show that $y \cbc x$ for all $x, y\in C_r^*(G)^+$. Fix  $x, y\in C_r^*(G)^+$.
Since $B(G)$ is dense in $C^*_r(G)$, we have  \[\overline{B(G)xB(G)} = \overline{\cs_r(G)x\cs_r(G)} = \cs_r(G).\]

 Fix $\epsilon > 0$. 
 Since $y$ is a positive element, so is $\sqrt{y}$. Since  $\sqrt{y}\in \overline{B(G)xB(G)}$, there exists a sequence $b_n\in B(G)xB(G)$ such that $b_n\to \sqrt{y}$.  Since $B(G)B(G)xB(G)\subseteq B(G)$ we have 
\[
b_nb_n^*\in B(G)xB(G)B(G)xB(G)\subseteq  B(G)xB(G)
\]
 and $b_nb_n^*\to y$. 
Thus there exists a positive $b\in B(G)xB(G)$ such that $\|b-y\|<\epsilon$.   Now \cite[Lemma~2.5(ii)]{KR} gives $(y-\epsilon)_+\cbc b$.
Since $\epsilon$ was fixed,  \cite[Proposition~2.6]{KR} gives
$y \cbc b$.  

We claim that $b\cbc x$. By the definition of $B(G)xB(G)$, there exist $c_i, d_i\in B(G)$ such that $b=\sum_{i=1}^k c_ixd_i$, and hence there exists a compact open subset $W$ of $\go$ such that $b=\sum_{i=1}^k 1_Wc_i1_Wx1_Wd_i1_W$. Thus $b\in B(G)1_Wx1_WB(G)$. 
Since $B(G)$ is s-unital (by Lemma~\ref{lem:B}) and is algebraically properly infinite (by Proposition~\ref{prop:Binf}), it follows from \cite[Lemma~3.4(i)]{AGPS} that $B(G)$ is algebraically purely infinite. Thus by Definition~\ref{defn ring} \eqref{b}, $b\in B(G)\cap B(G)1_Wx1_WB(G)$ implies $b\cba 1_Wx1_W$ in $B(G)$, that is, there exist $c, d\in B(G)$ such that \[b=c(1_Wx1_W)d=(c1_W)x(1_Wd).\] Now applying \cite[Proposition~2.4 (iii)$\Rightarrow$(ii)]{R92} to the positive elements $b, x$ of $C^*_r(G)$ and the constant sequences $c1_W, 1_Wd$, gives a sequence $\{r_n\}\subseteq C^*_r(G)$ such that $r_nxr_n^*\to b$. Thus $b\cbc x$. Since  $\cbc$ is transitive, we get $y\cbc x$. 
\end{proof}

\begin{lemma}
\label{lem:propinf}
If $1_U$ is a $\cs$-properly infinite projection in $\cs_r(G)$,  then $1_U$ is an algebraically properly infinite idempotent in $B(G)$.
\end{lemma}

\begin{proof}
Let $p := 1_U$ be an algebraically properly infinite projection in $\cs_r(G)$.  By the definition of algebraically properly infinite projection, there exist $q_1,q_2, x,y \in\cs_r(G)$ such that 
\[x^*x = p, \quad xx^* = q_1 \leq_{\cs} p, \quad y^*y = p, \quad \text{ and }\quad yy^*=q_2 \leq_{\cs} p-q_1.\]
Thus $x$ and $y$ are partial isometries in $\cs_r(G)$ with range projections $q_1$ and $q_2$, respectively, and source projection $p$.  Thus $xp =x$ and $q_1x=x$, and similarly for $y$.  Take $x':=pxp$ and $y' := pyp$.  Then $x',y' \in B(G)$.  
Now let $z$ be the 1 by 2 matrix $[x'\  y']$.  It is straightforward to check that
\[z^*z = p \oplus p \quad \text{and} \quad zz^* = q_1+q_2 \leq_{\cs} p. \]
Thus $p$ is an algebraically properly infinite idempotent in $B(G)$.
\end{proof}

\begin{lemma}
\label{lem:cbccba}
Let $G$ be a minimal and effective groupoid.  Suppose that $U,V \subseteq \go$ are nonempty compact open sets such that $1_U$ is an algebraically properly infinite idempotent in $B(G)$.  Then $1_V \cba 1_U$ in $B(G)$.
\end{lemma}

\begin{proof}                                                
Since $G$ is minimal and effective,  $A(G)$ is simple. Thus the ideal generated by $1_U$ in $A(G)$ is all of $A(G)$. 
Therefore there exist $a_i, b_i\in A(G)$ such that
\begin{align*}
1_V &= \sum_{i=1}^n a_i1_Ub_i\\
&\cba \bigoplus_{i=1}^n a_i1_Ub_i \text{ by \cite[Lemma~2.2(vi)]{AGPS} because $A(G)$ is s-unital }\\
&\cba \bigoplus_{i=1}^n 1_U \text{ by \cite[Lemma~2.2(ii)]{AGPS}}\\
&\cba 1_U \text{ in $B(G)$}
\end{align*} because $1_U$ is algebraically properly infinite in $B(G)$.
Finally, since $\cba$ is transitive by \cite[Lemma~2.2(i)]{AGPS} we have $1_V   \cba 1_U$ in  $B(G)$.      
\end{proof}

\begin{prop}
\label{prop:Binf}
 Let $G$ be a second-countable ample Hausdorff groupoid. Suppose that $G$ is minimal and effective, and that $\cs_r(G)$ is $\cs$-purely infinite.
Then the algebra \[B(G)= \{1_U x 1_U : x \in \cs_r(G) \text{ and } U \subseteq \go \text{ is a compact open set}\}\]  of Lemma~\ref{lem:B} is algebraically properly infinite.
\end{prop}

\begin{proof}
Since $G$ is minimal and effective, $\cs_r(G)$ is $\cs$-purely infinite simple.  Then \cite[Theorem~4.16]{KR} says that every positive element of $\cs_r(G)$ is $\cs$-properly infinite.  In particular, for every compact open $U \subseteq \go$, $1_U$ is a $\cs$-properly infinite projection in $\cs_r(G)$.  So Lemma~\ref{lem:propinf} says every $1_U$ is an algebraically properly infinite idempotent in $B(G)$.

Fix a nonzero $a \in B(G)$.  We need to show that $a$ is algebraically properly infinite, that is, that $a\oplus a\cba a$.   Since $a\in B(G)$, there exist a  compact open $V \subseteq \go$ and
$x \in \cs_r(G)$ such that $a = 1_Vx1_V$.  

We claim that there exists a compact open $U \subseteq \go$ such that 
$1_U \cba a^*a$ in $B(G)$.  Since $a^*a$ is positive, by the proof of  \cite[Lemma~3.2]{BCS} 
there exists a nonzero $h \in C_0(\go)^+$ such that $h \cba a^*a$ in $\cs_r(G)$.  (The statement of  \cite[Lemma~3.2]{BCS} gives $h$ such that $h\cbc a^*a$, but a look at the proof shows  $h\cba a^*a$.)
Since $h\in  C_0(\go)^+ \setminus \{0\}$ there exists a compact  open subset $U$ of $\go$ such that $h(u)>0$ for all $u\in U$. 
Define $g$ by 
\[g(u)=\begin{cases} \frac{1}{\sqrt{h(u)}} &\text{if $u\in U$}\\
0&\text{otherwise.} \end{cases}
\]
Then $g\in C_0(\go)$, and $ghg=1_U$.
Thus $1_U \cba h$ in $\cs_r(G)$.
Now the transitivity of $\cba$ gives $1_U \cba a^*a$ in $\cs_r(G)$.  In particular, there exists
$y,z\in \cs_r(G)$ such that $za^*ay=1_U$.  

Let $W:=U \cup V$. Then
$1_U = 1_W1_U1_W$ and $a^*a = 1_{V}x^*1_{V} 1_Vx1_V=1_Wa^*a1_W$,  and 
\[ 1_Wz1_Wa^*a1_Wy1_W =1_Wza^*ay1_W  =1_W1_U1_W=1_U.\]
Thus  $1_U \cba a^*a$ in $B(G)$, as claimed.

We have $a^*a = a^*a1_V$ and  $a^* = 1_{V}x^*1_{V} \in B(G)$, and hence 
\[a^*a=a^*a1_V=(1_{V}x^*1_{V})a (1_V).\]
Thus $a^*a \cba a$ in $B(G)$. Now   $1_U \cba a$ in $B(G)$
by transitivity of $\cba$.

Since $B(G)$ is s-unital and $1_U$ an algebraically properly infinite idempotent in $B(G)$, we apply \cite[Lemma~3.9(ii)]{AGPS}
to get that 
\[a \oplus 1_U \cba a\]
in $B(G)$.  We have  $a \cba 1_V$ because $a = (1_Vx1_V)1_V(1_V)$ and  Lemma~\ref{lem:cbccba} implies $1_V \cba 1_U$ in $B(G)$.
The transitivity of $\cba$ implies
$a \cba 1_U$ in $B(G)$.  We also have $a \cba a$ because $B(G)$ is s-unital, and so
\[a \oplus a \cba a \oplus 1_U\]
by \cite[Lemma~2.2(ii)]{AGPS}. Finally, transitivity  gives $a \oplus a \cba a$ in $B(G)$.
\end{proof}

Combining Propositions~\ref{prop:1} and \ref{prop:Binf} we get our main theorem.

\begin{thm}
\label{thm:main}
 Let $G$ be a second-countable ample Hausdorff  groupoid. Suppose that $G$ is minimal and effective.  Then  
$B(G)$ is  algebraically properly infinite if and only if  
$\cs_r(G)$ is $\cs$-purely infinite simple.
\end{thm}

Theorem~\ref{thm:main} together with Corollary~\ref{cor:main} give the following corollary.

\begin{cor}
Let $G$ be an ample Hausdorff groupoid.  Suppose that $G$ is effective and minimal.  If $A(G)$ is algebraically purely infinite, then $B(G)$ is algebraically properly infinite.
\end{cor}

The following remain  open questions:
\begin{questions}
 \begin{enumerate}
\item  Is there a groupoid $G$ such that $A(G)$ is algebraically purely infinite simple and $\cs_r(G)$ is not algebraically purely infinite?  
  \item Is there a groupoid $G$ such that $C^*_r(G)$ is $C^*$-purely infinite simple and $A(G)$ is not algebraically purely infinite simple?  (In other words, does the converse of Corollary~\ref{cor:main} fail?).    
\item  Let $U$ be a compact open subset of $\go$. If $1_U$ is an infinite projection in $\cs_r(G)$, is it an infinite idempotent in $A(G)$? This would be a kind of converse to Lemma~\ref{lem:referee}.
 \end{enumerate}
\end{questions}

\section{Kumjian--Pask algebras}
\label{section:KP}

In this section we show how our results apply to the subclass of algebras associated to higher-rank graphs.
$\cs$-purely infinite simple $\cs$-algebras of higher-rank graphs are considered 
in \cite{PRRS, Sims, EvansSims:JFA2012, CaHS, BCS, PSS}, but, in general, a condition on the graph that 
characterises pure infiniteness has remained elusive.

Kumjian--Pask algebras were defined in \cite{ACaHR} for higher-rank graphs without sources, 
and were then generalised to locally convex graphs in \cite{CFaH} and to finitely aligned graphs in \cite{CP}.  The groupoids  built from $k$-graphs are described, for example,  in \cite{KP, FMY, Y07, aHKR}.  To make our exposition as self contained as possible we will include some of this construction below.

We view $\N^k$ as a category with one object and composition given by addition.  A  \emph{higher-rank graph of rank $k$} or \emph{$k$-graph} is a countable category $\Lambda$ with a functor $d: \Lambda\to \N^k$, called the \emph{degree map},  
 that satisfies a {\em unique factorisation property}: 
 \begin{itemize}
 \item[] if $\lambda\in \Lambda$ with 
$d(\lambda)=m+n$, then there exist paths $\mu,\nu\in \Lambda$ such that $d(\mu)=m, d(\nu)=n$ and $\lambda=\mu\nu$.  
\end{itemize}
We call the domain and codomain of a morphism $\lambda$ the source and range of $\lambda$, respectively,  and write $s(\lambda)$ and $r(\lambda)$ for them. We call the  objects and morphisms in 
$\Lambda$ vertices and paths respectively; we denote the set of vertices by $\Lambda^0$ and the set of paths by $\Lambda$.  Thus $s, r:\Lambda\to\Lambda^0$.
For $v\in \Lambda^0$ and $n\in \N^k$, we set $v\Lambda:=r\inv(v)$,  $\Lambda^n:=d\inv(n)$ and $v\Lambda^n:=v\Lambda \cap \Lambda^n$. 

For $m=(m_1,\ldots, m_k)$ and $n=(n_1,\ldots, n_k)\in \N^k$, we define  
\begin{align*}                                              
m\wedge n &:=(\min\{m_1,n_1\}, \ldots, \min\{m_k,n_k\} ) \text{ and }\\
m \vee n &:= (\max\{m_1,n_1\}, \ldots, \max\{m_k,n_k\} ) .                  
\end{align*}
Let $\lambda,\mu \in \Lambda$. If $d(\lambda)=m+n$, we write $\lambda(0,m)$ and $\lambda(m,m+n)$ for the unique paths of degree $m$ and $n$ such that $\lambda= \lambda(0,m)\lambda(m,m+n)$. Then $
\tau\in\Lambda$ is a \emph{minimal common extension} of $\lambda$ and $\mu$ if
\begin{equation*}
d\left( \tau \right) =d\left( \lambda \right) \vee d\left( \mu \right), \ \tau \left( 0,d\left( \lambda \right) \right) =\lambda \text{ and }\tau
\left( 0,d\left( \mu \right) \right) =\mu .
\end{equation*}
We define
\begin{equation*}
\Lambda ^{\min }\left( \lambda ,\mu \right) :=\left\{ \left( \rho ,\tau
\right) \in \Lambda \times \Lambda :\lambda \rho =\mu \tau \text{\ is a minimal common extension of $\lambda,\mu$}%
 \right\} .
\end{equation*}
Then $\Lambda $ is \emph{finitely aligned} if $\Lambda ^{\min }\left(
\lambda ,\mu \right) $ is finite (possibly empty) for all $\lambda ,\mu \in
\Lambda $.  This class of $k$-graphs was introduced  
in \cite{RSY} and is the most general type of $k$-graph considered in the literature.

A set $E\subseteq v\Lambda $ is \emph{exhaustive} if for every $\lambda \in
v\Lambda $, there exists $\mu \in E$ such that $\Lambda ^{\min }\left(
\lambda ,\mu \right) \neq \emptyset $. The set of finite exhaustive sets is then
\begin{equation*}
\operatorname{FE}\left( \Lambda \right) =\bigcup_{v\in \Lambda ^{0}}\left\{
E\subseteq \left. v\Lambda \right\backslash \left\{ v\right\} :E\text{ is
finite and exhaustive}\right\}.
\end{equation*}

The unit space of the groupoid of a finitely aligned higher-rank graph 
consists of the set of boundary paths, defined as follows.
First, for $k\in \N$ and $m\in (\N\cup \{\infty\})^k$, let $\Omega_{k,m}$ be the category 
with objects $\{p\in \N^k: p\leq m\}$, morphisms 
\[\{(p,q): p,q\in \N^k, p\leq q\leq m\},\]
$r(p,q)=p$ and $s(p,q)=q$.  With degree functor $d(p,q)=q-p$, $\Omega_{k,m}$ is a $k$-graph.
A degree-preserving functor $x:\Omega _{k,m}\rightarrow \Lambda $ is a \emph{boundary path} of 
$\Lambda $ if for all $n\in \mathbb{N}^{k}$ with $n\leq m$ and all $E\in
x\left( n,n\right) \operatorname{FE}\left( \Lambda \right)$, there exists $\lambda
\in E$ such that $x\left( n,n+d\left( \lambda \right) \right) =\lambda $. We
write $\partial \Lambda $ for the set of all boundary paths.

For $p\in\mathbb{N}^k$ and $\lambda\in \Lambda^{p}$ we define $\sigma^p(\lambda)=\lambda(p, d(\lambda))$. As a set, the  groupoid of a finitely aligned higher-rank graph  is
\begin{align*}
 G_{\Lambda} :=\{&\left( x,m,y\right) \in \partial \Lambda \times \mathbb{Z}^{k}\times
\partial \Lambda :\text{there exists }p,q\in \mathbb{N}^{k}\text{ such that }
p\leq d\left( x\right) ,q\leq d\left(
y\right),\\
&p-q=m\text{ and }\sigma ^{p}(x)=\sigma ^{q}(y)\}\text{.}
\end{align*}
The range of $(x,m, y) \in G_{\Lambda}$ is $x$ and its source is $y$.  Composition is given by \[(x,m,y)(y, n,z)=(x,m+n,z),\] and a 
computation shows that this is well-defined. We identify $\partial \Lambda$  with $\go_\Lambda$ by $x\mapsto (x,0,x)$.

When endowed 
with the topology described in \cite{Y07}, $G_{\Lambda}$ is a second-countable, Hausdorff and ample groupoid.   
For  $\lambda \in \Lambda$ and finite subsets $F$ of $s(\lambda) \Lambda$, define
\begin{gather*}
Z(\lambda) := \{x \in \partial \Lambda: x(0,d(\lambda)) = \lambda\}\quad\text{and}\quad
Z(  \lambda \backslash F) :=
Z( \lambda ) \backslash \Big(\bigcup_{\nu \in
F}Z( \lambda \nu) \Big).
\end{gather*}
Then the unit space $\go_\Lambda$ has a basis of compact open sets of the form
$Z(  \lambda \backslash F)$.
For each $\lambda, \mu \in \Lambda$ with $s(\lambda) =s(\mu)$, 
define 
\[
Z( \lambda,\mu) :=\{( \lambda x,d(\lambda) -d(\mu) ,\mu x):
x \in \partial\Lambda \text{ and }s(\lambda) = s(\mu) = r(x) \}.
\]
Then the basis for the topology on $G_{\Lambda}$ consists of sets of the form 
\[
Z(\lambda, \mu \setminus F )
:=Z( \lambda,\mu ) \setminus
\Big(\bigcup_{\nu \in F}Z( \lambda \nu ,\mu \nu) \Big),
\] where again $F$ is a finite subset of $s(\lambda) \Lambda$. The subspace topology of $\go_\Lambda$  is the one described above.

Given a field $K$, we define the Kumjian--Pask algebra and the $\cs$-algebra of $\Lambda$  by 
\[\KP_K(\Lambda):=A_K(G_\Lambda) \quad\text{and}\quad \cs(\Lambda):=\cs_r(G_{\Lambda}),\]
respectively.
These definitions are consistent with previous definitions in the 
literature via the isomorphism characterised by $s_\mu\mapsto 1_{Z(\mu,s(\mu))}$ 
(\cite[Proposition~5.4]{CP} for $\KP_K(\Lambda)$ and \cite[Theorem~6.9]{FMY} for $C^*(\Lambda)$).   
To be consistent with the literature we write  
$s_v:=1_{Z(v,v)}$ and $s_\mu:=1_{Z(\mu, s(\mu))}$.  
We write $\KP(\Lambda)$ for $\KP_\C(\Lambda)$.

A $k$-graph $\Lambda$ is \emph{cofinal} if for all $v\in \Lambda ^{0}$ and $x\in \partial \Lambda $, 
there exists $n\leq d\left( x\right) $ such that $v\Lambda x(n)$ is nonempty.  
A $k$-graph is {\em aperiodic} if for all $v\in \Lambda^{0}$ there exists
$x\in v \partial \Lambda$ such that $\alpha x=\beta x$ implies $\alpha=\beta$. By \cite[Proposition~7.1]{CP} $G_\Lambda$ is minimal if and only 
if $\Lambda$ is cofinal, and by \cite[Proposition~6.3]{CP} $G_\Lambda$ is effective
if and only if $\Lambda$ is aperiodic.

\begin{cor}
\label{cor:KPPIS}
 Let $\Lambda$ be a finitely aligned higher-rank graph. Suppose that $\Lambda$ is cofinal and aperiodic.
 Then $\KP(\Lambda)$ is algebraically purely infinite simple if and only if $s_v$ is an infinite
 idempotent for all $v \in \Lambda^0$.
\end{cor}

\begin{proof}
Suppose that  $\KP(\Lambda)$ is algebraically purely infinite simple.  
Fix $v\in\Lambda^0$. Then $Z(v)$ is a compact open subset of  $\partial \Lambda=\go_{\Lambda}$. Thus $s_v=1_{Z(v)}$ is an infinite idempotent by Theorem~\ref{thm:pis}.

Conversely, suppose that $s_v$ is an infinite idempotent for all $v \in \Lambda^0$. We consider the basis  $\mathcal{B} = \{Z(\lambda\setminus F): \lambda \in \Lambda, \text{ finite } F \subseteq s(\lambda)\Lambda \}$ of compact open sets for $G_{\Lambda}^{(0)}$, and seek to apply Corollary~\ref{cor:basis}.
 First, consider a non-empty $Z(\lambda)$.  Then we have
 \[1_{Z(\lambda)}= s_{\lambda}s_{\lambda^*} \sima s_{\lambda^*}s_{\lambda} = s_{s(\lambda)} .\]
 By assumption, $s_{s(\lambda)}$ is an infinite idempotent.  Since $\KP(\Lambda)$ is s-unital, 
 $s_{\lambda}s_{\lambda^*}$ is an infinite idempotent by Lemma~\ref{lem:3.9}\eqref{it3:lem3.9}.
It follows that $1_{Z(\lambda)}$ is an infinite idempotent in  $\KP(\Lambda)$.

Second, consider a non-empty $Z(\lambda\setminus F)$. Then there exists $\lambda'\in s(\lambda)\Lambda$ such that $Z(\lambda\lambda')\subseteq Z(\lambda\setminus F)$. Then 
\[1_{Z(\lambda\lambda')}\leqa 1_{Z(\lambda\setminus F)},\]
and $1_{Z(\lambda\lambda')}$ is infinite by Lemma~\ref{lem:3.9}\eqref{it2:lem3.9}.
Thus  $\KP(\Lambda)$ algebraically purely infinite simple by Corollary~\ref{cor:basis}.
 \end{proof}

\begin{cor}
Let $\Lambda$ be a  finitely aligned higher-rank graph. Suppose that $\Lambda$  is cofinal and aperiodic.
\begin{enumerate}
\item\label{cor-1} If $\KP(\Lambda)$ is algebraically purely infinite simple, then $\cs(\Lambda)$ is $\cs$-purely infinite simple.
\item\label{cor-2}Suppose that $s_v$ is an infinite idempotent in $\KP(\Lambda)$ for all $v\in \Lambda^0$.
Then each $s_v$ is an infinite projection in $\cs(\Lambda)$.
\end{enumerate}
\end{cor}

\begin{proof} \eqref{cor-1} is immediate from Theorem~\ref{thm:main}.

 \eqref{cor-2}  Since  $s_v$ is an infinite idempotent in $\KP(\Lambda)$ for all $v\in \Lambda^0$, 
$\KP(\Lambda)$ is algebraically purely infinite simple by Corollary~\ref{cor:KPPIS}. Since $\KP(\Lambda)$ and $A(G_\Lambda)$ are isomorphic, 
Theorem~\ref{thm:main} implies that  $\cs(\Lambda)$ is $\cs$-purely infinite simple.  Then  every projection
 in $\cs(\Lambda)$ is an infinite projection by \cite[Corollary~5.1]{BCS}.
\end{proof}

We now connect our work with that of Larki  in \cite{Lar}. 
Recall from \cite[Lemma~3.2]{EvansSims:JFA2012} that
a pair $(\mu, \nu)$ is called a \emph{generalised cycle} if $Z(\mu) \subseteq Z(\nu)$.
 This generalised cycle has an \emph{entrance} if the containment is strict.  

Let $\Lambda$ be a finitely aligned,  aperiodic and cofinal $k$-graph such that every vertex of $\Lambda$ 
can be reached from a generalised cycle with an entrance. Then \cite[Theorem~5.4]{Lar} implies that $\KP(\Lambda)$
is algebraically purely infinite simple.  The next lemma shows that this also follows from Corollary~\ref{cor pis}.  

\begin{lemma} Let $\Lambda$ be a finitely aligned higher-rank graph. Suppose  that every vertex of  $\Lambda$ can be reached from a generalised cycle with an entrance. Then $G_\Lambda$ is locally contracting.
\end{lemma}
\begin{proof}  Fix a nonempty compact open $U \subseteq \go$. Then there exists $\lambda\in \Lambda$ and a 
finite subset $F$ of $s(\lambda)\Lambda$ such that $\emptyset\neq Z(\lambda\setminus F)\subseteq U$.  
Further, since every vertex can be reached from a generalised cycle, $Z(\lambda\setminus F) \neq \{\lambda\}$.
Thus there exists $\lambda'\in s(\lambda)\Lambda$ such that $d(\lambda') \neq 0$ and
$Z(\lambda\lambda')\subseteq Z(\lambda\setminus F)$. 
By assumption, $s(\lambda')$ can be reached from a generalised cycle $(\mu, \nu)$ with an entrance. 
Say $\lambda\lambda'\beta\mu\in\Lambda$. 
Then $(\lambda\lambda'\beta\mu,\lambda\lambda'\beta\nu)$ is a generalised cycle with an entrance.  
Then 
\[B=Z(\lambda\lambda'\beta\mu,\lambda\lambda'\beta\nu)\] 
 is a compact
open bisection with \[s(B)=Z(\lambda\lambda'\beta\nu)\subsetneq Z(\lambda\lambda'\beta\mu)=r(B)\subseteq U.\]
Thus $G_\Lambda$ is locally contracting. 
\end{proof}

We now restrict to row-finite locally convex $k$-graphs. Let $\Lambda$ be a $k$-graph. Then $\Lambda$ is \emph{row-finite} if $v\Lambda^n$ is finite for all $v\in \Lambda^0$ and $n\in \N^k$;  $\Lambda$ is \emph{locally
convex} if for every distinct $i,j\in \{1,....,k\}$, $v\in \Lambda^0$ and paths $\lambda\in v\Lambda^{e_i}, \mu\in v\Lambda^{e_j}$,  the sets $s(\lambda)\Lambda^{e_j}$ and $s(\mu)\Lambda^{e_i}$ are nonempty. 
Since our graphs might have sources, for $n\in \N$ we define 
 \[
 \Lambda^{\leq n}:=\{\lambda\in \Lambda: d(\lambda)\leq n \text{   and   } s(\lambda)\Lambda^{e_i}=\emptyset \text{  whenever  } d(\lambda)+e_i \leq n\}. \]
When $\Lambda$ is locally convex, by \cite[Proposition~2.12]{Webster} the boundary path space $\partial \Lambda$ equals 
\begin{align*}\Lambda^{\leq \infty} := 
&\{x: \Omega_{k,m} \to \Lambda: \text{ $x$ is a functor such that if }
 v \in  \Omega_{k,m}^0 \text{ and }  v\Omega_{k,m}^{\leq e_i} = \{v\},\\ & \text{ then }
x(v)\Lambda^{\leq e_i} = \{x(v)\}\}.\end{align*}

 A path $\alpha\in \Lambda\setminus \Lambda^0$ is a {\em return path} if $r(\alpha)=s(\alpha)$; a return path $\alpha$ 
 has an {\em entrance} if there exists a $\delta\in s(\alpha)\Lambda$ with $d(\delta)\leq d(\alpha)$ 
 and $\alpha(0,d(\delta))\neq \delta$. 
 Our aim  is to prove the  following:

\begin{prop}\label{prop:entrance} Let $\Lambda$ be a  
row-finite locally convex higher-rank graph. Suppose that $\Lambda$ is cofinal and aperiodic, and that $\Lambda$ contains a return path with an entrance.  Then
 each vertex in $\Lambda$ can be reached from a return path with an entrance.  
\end{prop}

We then get:

\begin{cor}
\label{cor:KP}
Let $\Lambda$ be a  row-finite locally convex higher-rank graph. Suppose that $\Lambda$   is  
cofinal and aperiodic, and   that $\Lambda$ contains a return path with an entrance. 
Then $\KP(\Lambda)$ and $C^*(\Lambda)$ are algebraically and $C^*$-purely infinite simple,  respectively.
\end{cor}

\begin{proof}
Let $\mu$ be a return path with an entrance. Then  $(\mu\mu, \mu)$ is a generalised cycle with 
an entrance. Then Proposition~\ref{prop:entrance} implies that every vertex can be reached from a generalised cycle. 
By  \cite[Theorem~5.4]{Lar}, $\KP(\Lambda)$ is  algebraically purely infinite simple. 
Then $C^*(\Lambda)$ is $C^*$-purely infinite simple by Theorem~\ref{thm:main}.
\end{proof}

For $1$-graphs, Proposition~\ref{prop:entrance}  follows directly from cofinality:   if $\alpha$ is a return 
path with an entrance, then $x=\alpha\alpha\cdots$ is a boundary path,  and  cofinality implies that for every 
$v\in \Lambda^0$ there exists $\mu_v \in v\Lambda$ such that $s(\mu_v)$ is a vertex on $x$ and hence on $\alpha$.  
When we try to apply this idea to a general $k$-graph it fails: $x=\alpha\alpha\cdots$ may not be a boundary path.   
In fact, $x$ is not a boundary path whenever  $d(\alpha)\wedge (1,1,\ldots,1)\neq (1,1,\ldots, 1)$. The keys to proving  Proposition~\ref{prop:entrance} are Lemma~\ref{lem:ret} and  Corollary~\ref{infinitepathwithentrance}, 
where we construct a boundary path $x$ that has enough return paths with entrances.

\begin{lemma}\label{lem:ret} Let $\Lambda$ be a row-finite and locally convex higher-rank graph.  Suppose that $\Lambda$ has  a return path $\alpha$ with an entrance and  that there exists $\mu\in s(\alpha)\Lambda\setminus\{s(\alpha)\}$ with 
$d(\mu)\wedge d(\alpha)=0$.  Then there exist $P\in \N$, a return path $\beta$ with an entrance  such that  $d(\beta)=Pd(\alpha)$,  and a path from $s(\beta)$ to $s(\alpha)$ of degree at least $d(\mu)$.  
\end{lemma}

\begin{proof}
Let $\alpha$ be a return path with an entrance $\delta$. We may    
extend $\delta$ until $\delta\in s(\alpha)\Lambda^{\leq d(\alpha)}$.   
By assumption there exists  $\mu\in s(\alpha)\Lambda\setminus \{s(\alpha)\}$ with $d(\mu)\wedge d(\alpha)=0$.

By local convexity applied to the pair $(\mu,  \delta)$, there  exists $\mu_0\in s(\delta)\Lambda^{d(\mu)}$. If 
$s(\delta)=r(\delta)$ we can take $\mu_0=\mu$.  We will now use the unique factorisation property to construct the path $\beta$.

Consider $\delta\mu_0$.  Applying the unique factorisation property to the path $\delta\mu_0$ gives $\mu_1\in s(\alpha) \Lambda^{d(\mu)}$ and $\gamma_1\in \Lambda^{\leq d(\alpha)} s(\mu_0)$ such that 
\[
\delta\mu_0=\mu_1\gamma_1.
\]
By the unique factorisation property and induction,  for $i\geq 2$ there exist $\mu_i\in s(\alpha)\Lambda^{d(\mu)}$ and $\gamma_i\in \Lambda^{d(\alpha)}s(\mu_{i-1})$ such  that
\begin{equation}\label{constructingpath}
\alpha \mu_{i-1}=\mu_i\gamma_i.
\end{equation}
Notice that $\{\mu_i: i\geq 0\}\subseteq s(\alpha)\Lambda^{d(\mu)}$ is finite because $\Lambda$ is row-finite.
So there exists $i\neq j$ such that $\mu_i=\mu_j$, and in particular $s(\mu_i)=s(\mu_j)$. 
Let $m,n\in\N$ with $m<n$ be such that $(m,n)$ is   smallest with respect to the dictionary order and  such that  $\mu_m=\mu_n$. Then
\[\theta:=\gamma_{n}\gamma_{n-1}\cdots\gamma_{m+1}\]
is a return path.  

First, suppose that $m\neq 0$. By the minimality of $(m, n)$,  $\gamma_{m}\neq \gamma_i$ for  $m+1\leq i\leq n$, and hence $\gamma_{m}$ is an entrance to $\theta$.  Notice that $\delta\mu_0\gamma_1\dots\gamma_m$ is a path of degree more than $d(\mu)$ from $s(\theta)$ to $s(\alpha)$.
Since $d(\gamma_i)=d(\alpha)$ for $i\geq 2$, we get that   $\beta:=\theta$ with $P=m-n$ satisfies the lemma. 

Second, suppose that $m=0$. Then 
\[\theta=\gamma_{n}\cdots \gamma_1.\]   
Here $\theta$ is a return path, but it may not have an entrance. We will  construct another return path with an entrance as required.  Since $\mu_n=\mu_0$, we have $s(\alpha)=r(\delta)=s(\delta)$. In particular, $d(\delta)=d(\alpha)$ (for otherwise we could extend $\delta\in \Lambda^{\leq d(\alpha)}$ further using $\alpha$).

We claim that there exist $q\in \N$ and $1\leq i\leq n$ such that $\gamma_{i+qn}\neq \gamma_{i}$. Aiming for a contradiction, we suppose that $\gamma_{i+nq}=\gamma_{i}$ for all   $q\in\N$ and $1\leq i\leq n$. 
We will show by induction on $q$ that 
\[ \Gamma_q:=\{\mu_1,\mu_{1+n},\ldots, \mu_{1+qn}\}\] has
$q+1$ distinct elements. Then $\Gamma_q \subseteq s(\alpha)\Lambda^{d(\mu)}$ for every $q$, and this contradicts that $s(\alpha)\Lambda^{d(\mu)}$ is finite.
When $q=0$ we have $\Gamma_q=\{\mu_1\}$ and there is nothing to prove. For the inductive step, we  will use the following two observations. 
\begin{enumerate}
\item\label{obs1} $\mu_1\neq \mu_{1+tn}$ for any $t\in \N$ (for otherwise \eqref{constructingpath} gives $\alpha\mu_{tn}=\mu_{1+tn}\gamma_{1+tn}=\mu_1\gamma_1=\delta\mu_0$, contradicting that $\alpha\neq \delta$).
\item\label{obs2} Let $1\leq i\leq n$. Then  \eqref{constructingpath} gives
\[
\mu_{i+1+tn}\gamma_{i+1}=\mu_{i+1+tn}\gamma_{i+1+tn}=\alpha\mu_{i+tn}.
\]
So if $\{\mu_{i+tn}:0\leq t\leq q \}$ has $q+1$ elements, then $\{\mu_{i+1+tn}:0\leq t\leq q \}$ has $q+1$ elements.
\end{enumerate}
Now suppose that $q\geq 0$ and that $\Gamma_q$ has $q+1$ elements. We apply the observation~(\ref{obs2}) once to get that
\[
\{\mu_2, \mu_{2+n}, \dots, \mu_{2+qn} \}
\]
has $q+1$ elements. After $n-1$ further applications of observation~(\ref{obs2}) we get that 
\[
\{\mu_{1+q}, \mu_{1+2q}, \dots, \mu_{1+(q+1)n}\}
\]
has $q+1$ elements. Now we use observation~(\ref{obs1}) to see that 
\[
\{\mu_1, \mu_{1+q}, \mu_{1+2q}, \dots, \mu_{1+(q+1)n}\}
\]
has $q+2$ elements. This proves our claim. 

Let $(q_0,i_0)$ be smallest in  the dictionary order such that $\gamma_{i_0+q_0n}\neq \gamma_{i_0}$.
Since  $\{\mu_i:i\geq i_0+q_0n\}$ is finite, there exist smallest $(i,j)$  such that $i_0+q_0n\leq i<j$ and $s(\mu_i)=s(\mu_j)$.  Then
\[
\beta:=\gamma_j\dots\gamma_{i+1}
\]
is a return path of degree $(j-1)d(\alpha)$. If $i>i_0+q_0n$, then $\gamma_{i-1}$ is an entrance to $\beta$. If $i=i_0+q_0n$, then $\gamma_{i_0-1}$ is an entrance. In either case there is a path from $s(\beta)$ to $s(\alpha)$ of degree at least $d(\mu)$.
\end{proof}

\begin{cor}\label{infinitepathwithentrance} Let $\Lambda$ be a row-finite and
locally convex $k$-graph. Suppose that $\Lambda$ contains a return path with an entrance.  Then there 
exists $x\in \Lambda^{\leq \infty}$ such that for all $m\leq d(x)$ there exists $n\geq m$ 
and a $p\in\N^k$ such that $x(n, n+p)$ is a return path with an entrance.  
\end{cor}

\begin{proof}
Let $\alpha$ be a return path with an entrance.  We construct a boundary path by induction.
If $\alpha\alpha\cdots\in \Lambda^{\leq \infty}$ we are done.  
So suppose that $\alpha\alpha\cdots\not\in \Lambda^{\leq \infty}$.  Let  
\[t=(1,\ldots,1)-(1,\ldots, 1)\wedge d(\alpha).\] 
Since $\alpha\alpha\cdots\notin\Lambda^{\leq \infty}$,  there exists
$\mu_1\in s(\alpha)\Lambda^{\leq t}$.
Using Lemma~\ref{lem:ret} we find $P_1\in \N$ and a return 
path $\beta_1$ with an entrance  such that, without loss of 
generality, $\beta_1 \in  s(\mu_1)\Lambda^{P_1d(\alpha)}$.

Let $i\geq 1$, and suppose that  we have $\mu_i\in \Lambda^{\leq t}$, $P_i$ and $\beta_i\in s(\mu_i)\Lambda^{P_i d(\alpha)} s(\mu_i)$ such that $\beta_i$  has an entrance. If $\beta_i\beta_i\cdots\in \Lambda^{\leq\infty}$, then take $x=\beta_i\beta_i\cdots$. 
Otherwise there exists $\mu_{i+1}\in s(\beta_i)\Lambda^{\leq t}$.  Using  Lemma~\ref{lem:ret} applied to $\beta_i$ there exists 
$P'_{i+1}\in \N$ and a return path $\beta_{i+1}\in s(\mu_{i+1})\Lambda^{P_{i+1}'P_id(\alpha)}$ with an entrance. Take $P_{i+1}=P_{i+1}'P_i$.  

If this process terminates at $\beta_{i_0}$ take $x=\beta_{i_0}\beta_{i_0}\cdots$.
If this process never terminates take 
\[x=\alpha\mu_1\beta_1\mu_2\beta_2\cdots.\]
In either case, $x\in \Lambda^{\leq \infty}$ has the desired properties.  
\end{proof}

\begin{remark}
Corollary~\ref{infinitepathwithentrance} fills a gap in the 
proof of \cite[Corollary~5.7]{EvansSims:JFA2012} and in the last statement of \cite[Proposition~8.8]{Sims}.  
The last statement of \cite[Proposition~8.8]{Sims} claims that if $\Lambda$ is cofinal and contains a return path with an entrance, then every  $v\in \Lambda^0$ can be reached  from a return path with an entrance. Our Corollary~\ref{infinitepathwithentrance} and cofinality ensure this is indeed 
 the case. For  \cite[Corollary~5.7]{EvansSims:JFA2012},  Evans and Sims construct a single return path with entrance and then apply  \cite[Proposition~8.8]{Sims}.
\end{remark}

Now we are ready to prove Proposition~\ref{prop:entrance} 
\begin{proof}[Proof of Proposition~\ref{prop:entrance}]
 Since $\Lambda$ contains a return path with an entrance, 
Corollary~\ref{infinitepathwithentrance} gives  $x\in \Lambda^{\leq \infty}$ such that for 
all $m\leq d(x)$ there exists $n\geq m$ and $p$ such that $x(n, n+p)$ is a return path with an 
entrance.  Let $v\in \Lambda^0$. Since $\Lambda$ is cofinal there exists  $m\leq d(x)$ such 
that $v\Lambda x(m)$ is nonempty.  Pick $\omega_v\in v \Lambda x(m)$.  
Now pick $n\geq m$ and $p$ such that $x(n, n+p)$ is a return path with an entrance.  
Then $\mu_v=\omega_v x(m,n)$ connects $v$ to a return path with an entrance as desired.
\end{proof}

\end{document}